\documentclass[12pt]{article}
\usepackage{color}
\usepackage{amsthm}
 \usepackage{amsmath}
 \usepackage{amssymb}
\usepackage{amsfonts}
\usepackage{amsbsy}
\usepackage{multicol}
\usepackage[dvips]{graphics}
\usepackage{graphics,pspicture}
\newcommand{\be}{\begin{equation}}
\newcommand{\ee}{\end{equation}}
\newcommand{\bea}{\begin{eqnarray}}
\newcommand{\eea}{\end{eqnarray}}
\newcommand{\ba}{\begin{array}}
\newcommand{\ea}{\end{array}}

\newcommand{\bc}{\begin{center}}
\newcommand{\ec}{\end{center}}
\newcommand{\ben}{\begin{enumerate}}
\newcommand{\een}{\end{enumerate}}
\newcommand{\bfi}{\begin{figure}}
\newcommand{\efi}{\end{figure}}

\newcommand{\bq}{\begin{quote}}
\newcommand{\eq}{\end{quote}}
\newcommand{\bqu}{\begin{quotation}}
\newcommand{\equ}{\end{quotation}}
\newenvironment{emphit}{\begin{itemize}}{\end{itemize}}
\newcommand{\bemp}{\begin{emphit}}
\newcommand{\eemp}{\end{emphit}}

\newcommand{\bt}{\begin{tabular}}
\newcommand{\et}{\end{tabular}}

\newtheorem{myth}{Theorem}[section]
\newtheorem{mylem}{Lemma}[section]

\newtheorem{mydef}{Definition}[section]
\newtheorem{myrem}{Remark}[section]

\begin{document}
\date{}
\title{New Algebraic Properties of Middle Bol Loops
}
\author{T. G. Jaiy\'e\d ol\'a\thanks{All correspondence to be addressed to this author.} \\
Department of Mathematics,\\
Obafemi Awolowo University,\\
Ile Ife 220005, Nigeria.\\
jaiyeolatemitope@yahoo.com\\tjayeola@oauife.edu.ng \and
S. P. David \\
Department of Mathematics,\\
Obafemi Awolowo University,\\
Ile Ife 220005, Nigeria.\\
davidsp4ril@yahoo.com\\
davidsp4ril@gmail.com\and
Y. T. Oyebo\\
Department of Mathematics, Lagos State University, Ojo, Lagos State, Nigeria\\
oyeboyt@yahoo.com,~yakub.oyebo@lasu.edu.ng}\maketitle
\begin{abstract}
A loop $(Q,\cdot,\backslash,/)$ is called a middle Bol loop if it obeys the identity $x(yz\backslash x)=(x/z)(y\backslash x)$. In this paper, some new algebraic properties of a middle Bol loop are established. Four bi-variate mappings $f_i,g_i,~i=1,2$ and four $j$-variate mappings $\alpha_j,\beta_j,\phi_j,\psi_j,~j\in\mathbb{N}$ are introduced and some interesting properties of the former are found. Neccessary and sufficient conditons in terms of $f_i,g_i,~i=1,2$, for a middle Bol loop to have the elasticity property, RIP, LIP, right alternative property (RAP) and left alternative property (LAP) are establsihed. Also, neccessary and sufficient conditons in terms of $\alpha_j,\beta_j,\phi_j,\psi_j,~j\in\mathbb{N}$, for a middle Bol loop to have power  RAP and power LAP are establsihed. Neccessary and sufficient conditons in terms of $f_i,g_i,~i=1,2$ and $\alpha_j,\beta_j,\phi_j,\psi_j,~j\in\mathbb{N}$, for a middle Bol loop to be a group, Moufang loop or extra loop are established. A middle Bol loop is shown to belong to some classes of loops whose identiites are of the J.D. Phillips' RIF-loop and WRIF-loop (generalizations of Moufang and Steiner loops) and WIP power associative conjugacy closed loop types if and only if some identities defined by $g_1$ and $g_2$ are obeyed.
\end{abstract}
\section{Introduction}
\paragraph{}
Let $G$ be a non-empty set. Define a binary operation ($\cdot $) on
$G$. If $x\cdot y\in G$ for all $x, y\in G$, then the pair $(G, \cdot )$
is called a \textit{groupoid} or \textit{Magma}.

If each of the equations:
\begin{displaymath}
a\cdot x=b\qquad\textrm{and}\qquad y\cdot a=b
\end{displaymath}
has unique solutions in $G$ for $x$ and $y$ respectively, then $(G,
\cdot )$ is called a \textit{quasigroup}.

If there exists a unique element $e\in G$ called the
\textit{identity element} such that for all $x\in G$, $x\cdot
e=e\cdot x=x$, $(G, \cdot )$ is called a \textit{loop}. We write
$xy$ instead of $x\cdot y$, and stipulate that $\cdot$ has lower
priority than juxtaposition among factors to be multiplied. For
instance, $x\cdot yz$ stands for $x(yz)$.

Let $x$ be a fixed element in a groupoid $(G, \cdot )$. The left and
right translation maps of $G$, $L_x$ and $R_x$ respectively can be
defined by
\begin{displaymath}
yL_x=x\cdot y\qquad\textrm{and}\qquad yR_x=y\cdot x.
\end{displaymath}
It can be seen that a groupoid $(G, \cdot )$ is a quasigroup if
it's left and right translation mappings are bijections or
permutations. Since the left and right translation mappings of a
loop are bijective, then the inverse mappings $L_x^{-1}$ and
$R_x^{-1}$ exist. Let
\begin{displaymath}
x\backslash y =yL_x^{-1}=y\mathcal{L}_x=x\mathcal{R}_y\qquad\textrm{and}\qquad
x/y=xR_y^{-1}=x\mathbb{R}_y=y\mathbb{L}_x
\end{displaymath}
and note that
\begin{displaymath}
x\backslash y =z\Longleftrightarrow x\cdot
z=y\qquad\textrm{and}\qquad x/y=z\Longleftrightarrow z\cdot y=x.
\end{displaymath}
Hence, $(G, \backslash )$ and $(G, /)$ are also quasigroups. Using
the operations ($\backslash$) and ($/$), the definition of a loop
can be stated as follows.

\begin{mydef}\label{0:1}\textrm{
A \textit{loop} $(G,\cdot ,/,\backslash ,e)$ is a set $G$ together
with three binary operations ($\cdot $), ($/$), ($\backslash$) and
one nullary operation $e$ such that
\begin{description}
\item[(i)] $x\cdot (x\backslash y)=y$, $(y/x)\cdot x=y$ for all
$x,y\in G$,
\item[(ii)] $x\backslash (x\cdot y)=y$, $(y\cdot x)/x=y$ for all
$x,y\in G$ and
\item[(iii)] $x\backslash x=y/y$ or $e\cdot x=x$ for all
$x,y\in G$.
\end{description}}
\end{mydef}
We also stipulate that ($/$) and ($\backslash$) have higher priority
than ($\cdot $) among factors to be multiplied. For instance,
$x\cdot y/z$ and $x\cdot y\backslash z$ stand for $x(y/z)$ and
$x(y\backslash z)$ respectively.

In a loop $(G,\cdot )$ with identity element $e$, the \textit{left
inverse element} of $x\in G$ is the element $xJ_\lambda
=x^\lambda\in G$ such that
\begin{displaymath}
x^\lambda\cdot x=e
\end{displaymath}
while the \textit{right inverse element} of $x\in G$ is the element
$xJ_\rho =x^\rho\in G$ such that
\begin{displaymath}
x\cdot x^\rho=e.
\end{displaymath}
For an overview of the theory of
loops, readers may check \cite{davidref:7,davidref:8,davidref:9,davidref:16,davidref:10,davidref:12,davidref:21,davidref:22}.

A loop $(G, \cdot )$ is said to be a power associative loop if $<x>$ is a
subgroup for all $x\in G$ and a diassociative loop if $<x, y>$ is a
subgroup for all $x, y\in G$.

\begin{mydef}
Let $(G, \cdot )$ be a loop. $G$ is said to be a left alternative property loop (LAPL) if
for all $x, y\in G,~ x\cdot xy=xx\cdot y$, a right alternative property loop (RAPL)if
for all $x, y\in G,~ yx\cdot x=y\cdot xx$, and an alternative loop if it is both left and right alternative.

A power associative loop $(G, \cdot )$ is said to be a power left alternative property loop (PLAPL) if
for all $x, y\in G,~ \underbrace{x(\cdots (x(x}_{n-\textrm{times}}y))\cdots )=x^ny$ and a power right alternative property loop (PRAPL)if
for all $x, y\in G,~ (\cdots((y\underbrace{x)x)\cdots )x}_{n-\textrm{times}}=yx^n$.

A loop $(G,\cdot )$ is called a flexible or an elastic loop if the
flexibility or elasticity property
\begin{displaymath}
xy\cdot x=x\cdot yx
\end{displaymath}
holds for all $x,y\in G$.

$(G, \cdot )$ is said to have the left inverse property (LIP) if
for all $x, y\in G,~  x^\lambda\cdot xy=y$, the right inverse property (RIP) if
for all $x, y\in G,~ yx\cdot x^\rho=y$ and the inverse property if it has both left and right inverse properties.
\end{mydef}

There are some classes of loops which do not have the inverse
property but have properties which can be considered as variations
of the inverse property.

A loop $(G,\cdot )$ is called a weak inverse property loop (WIPL) if
and only if it obeys the identity
\begin{equation}
x(yx)^\rho=y^\rho\qquad\textrm{or}\qquad(xy)^\lambda x=y^\lambda
\end{equation}
for all $x,y\in G$.
\begin{mydef}
A loop $(G,\cdot )$ is called a cross inverse property loop(CIPL) if it obeys the identity
\begin{equation}
xy\cdot x^\rho =y\qquad\textrm{or}\qquad x\cdot yx^\rho
=y\qquad\textrm{or}\qquad x^\lambda\cdot
(yx)=y\qquad\textrm{or}\qquad x^\lambda y\cdot x=y
\end{equation}
for all $x,y,\in G$.

A loop $(G,\cdot )$ is called an automorphic inverse property
loop(AIPL) if it obeys the identity
\begin{equation}
(xy)^\rho=x^\rho y^\rho\qquad\textrm{or}\qquad(xy)^\lambda
=x^\lambda y^\lambda
\end{equation}
for all $x,y,\in G$.

A loop $(G,\cdot )$ is called an anti-automorphic inverse property
loop(AAIPL) if it obeys the identity
\begin{equation}
(xy)^\rho= y^\rho x^\rho\qquad\textrm{or}\qquad(xy)^\lambda
=y^\lambda x^\lambda
\end{equation}
for all $x,y,\in G$.

A loop $(G,\cdot )$ is called a semi-automorphic inverse property
loop(SAIPL) if it obeys the identity
\begin{equation}
(xy\cdot x)^\rho= x^\rho y^\rho\cdot
x^\rho\qquad\textrm{or}\qquad(xy\cdot x)^\lambda =x^\lambda
y^\lambda\cdot x^\lambda
\end{equation}
for all $x,y,\in G$.
\end{mydef}

A loop satisfying the identical relation
 \begin{equation}\label{davideq1}
 (xy\cdot z)y=x(yz\cdot y)
\end{equation}
is called a right Bol loop (Bol loop). A loop satisfying the identical relation
\begin{equation}\label{davideq2}
(x\cdot yx)z=x(y\cdot xz)
\end{equation}
is called a left Bol loop.

A loop $(Q,\cdot)$ is called a middle Bol if it satisfies the identity
\begin{equation}\label{davideq2.1}
x(yz\backslash x) = (x/z)(y\backslash x)
\end{equation}

It is known that the identity \eqref{davideq2.1} is universal under loop isotopy and that the universality of \eqref{davideq2.1} implies the power associativity of the middle Bol loops (Belousov \cite{davidref:24}). Furthermore, \eqref{davideq2.1} is a necessary and sufficient condition for the universality of the anti-automorphic inverse property (Syrbu \cite{davidref:6}). Middle Bol loops were originally introduced in 1967 by Belousov \cite{davidref:24} and were later considered in 1971 by Gwaramija \cite{davidref:23}, who proved that a loop $(Q,\circ)$ is middle Bol if and only if there exists a right Bol loop $(Q,\cdot)$ such that
\begin{equation}\label{davideq2.2}
x\circ y=(y\cdot xy^{-1})y,~\textrm{for every}~x,y\in Q.
\end{equation}
This result of Gwaramija \cite{davidref:23} is formally stated below:
\begin{myth}\label{davidthm1}
If $(Q,\cdot)$ is a left (right) Bol loop then the groupoid $(Q,\circ)$, where $x\circ y=y(y^{-1}x\cdot y)$ (respectively, $x\circ y=(y\cdot xy^{-1})y$), for all $x,y \in Q$ , is a middle Bol loop and, conversely, if $(Q,\circ)$ is a middle Bol loop then there exists a left(right) Bol loop $(Q,\cdot)$ such that $x\circ y=y(y^{-1}x\cdot y)$ (respectively, $x\circ y=(y\cdot xy^{-1})y$), for all $x,y \in Q$.
\end{myth}
\begin{myrem}
Theorem~\ref{davidthm1} implies that if $(Q,\cdot)$ is a left Bol loop and $(Q,\circ)$ is the corresponding middle Bol loop then $x\circ y=x/y^{-1}$ and $x\cdot y=x//y^{-1}$ , where "$/$" ("$//$") is the right division in $(Q,\cdot)$ (respectively, in $(Q,\circ)$). Similarly, if $(Q,\cdot)$ is a right Bol loop and $(Q,\circ)$ is the corresponding middle Bol loop then $x\circ y=y^{-1}\backslash y$ and $x\cdot y=y//x^{-1}$ , where "$\backslash$" ("$//$") is the left (right) division in $(Q,\cdot)$ (respectively, in $(Q,\circ)$). Hence, a middle Bol loops are isostrophs of left and right Bol loops.\\

If $(Q,\circ)$ is a middle Bol loop and $(Q,\cdot)$ is the corresponding left Bol loop, then $(Q,\ast)$, where $x\ast y=y\cdot x$, for every $x,y\in Q$, is the corresponding right Bol loop for $(Q,\circ)$. So, $(Q,\cdot)$ is a left Bol loop, $(Q,\ast)$ is a right Bol loop and
$$x\circ y=y(y^{-1} x\cdot y)=[y\ast(x\ast y^{-1})]\ast y,$$
for every $x,y\in Q$.
\end{myrem}

After then, middle Bol loops resurfaced in literature not until 1994 and 1996 when Syrbu \cite{davidref:4,davidref:5} considered them in-relation to the universality of the elasticity law.

In 2003, Kuznetsov \cite{davidref:22}, while studying gyrogroups (a special class of Bol loops) established some algebraic properties of middle Bol loop and designed a method of constructing a middle Bol loop from a gyrogroup. According to Gwaramija \cite{davidref:23}, in a middle Bol loop $(Q,\cdot)$ with identity element $e$, the following are true.
\begin{enumerate}
    \item The left inverse element $x^\lambda$ and the right inverse $x^\rho$ to an element $x\in Q$ coincide : $x^\lambda$=$x^\rho$.
  \item If $(Q,\cdot,e)$ is a left Bol loop and "$/$" is the right inverse operation to the operation $"\cdot "$ , then the operation  $x\circ y=x/y^{-1}$ is a middle Bol loop $(Q,\circ,e)$, and every one middle Bol loop can be obtained in a similar way from some left Bol loop.
\end{enumerate}
These confirm the observations of earlier authors mentioned above.

In 2010, Syrbu \cite{davidref:6} studied the connections between structure and properties of middle Bol loops and of the corresponding left Bol loops. It was noted that two middle Bol loops are isomorphic if and only if the corresponding left (right) Bol loops are isomorphic, and a general form of the autotopisms of middle Bol loops was deduced. Relations between different sets of elements, such as nucleus, left (right,middle) nuclei, the set of Moufang elements, the center, e.t.c. of a middle Bol loop and  left Bol loops were established.

In 2012, Grecu and Syrbu \cite{davidref:65} proved that two middle Bol loops are isotopic if and only if the corresponding right (left) Bol loops are isotopic. They also proved that a middle Bol loop $(Q,\circ)$ is flexible if and only if the corresponding right Bol loop $(Q,\cdot)$ satisfies the identity
$$(yx)^{-1}\cdot \big(x^{-1}\cdot y^{-1}\big)^{-1}x=x.$$

In 2012, Drapal and Shcherbacov \cite{davidref:1} rediscovered the middle Bol identities in a new way.

In 2013,  Syrbu and Grecu \cite{davidref:66n} established a necessary and sufficient condition  for the quotient loops of a  middle Bol loop and of its corresponding right Bol loop to be isomorphic.

In 2014, Grecu and Syrbu \cite{davidref:66} established:
\begin{enumerate}
  \item that the commutant (centrum) of a middle Bol loop is an AIP-subloop and
  \item a necessary and sufficient condition when the commutant is an invariant under the existing isostrophy between middle Bol loop and the corresponding right Bol loop.
\end{enumerate}

In 1994, Syrbu \cite{davidref:4}, while studying loops with universal elasticity $(xy\cdot x=x\cdot yx)$ established a necessary and sufficient condition  $(xy/z)(b\backslash xz)=x(b\backslash[(by/z)(b\backslash xz)])$ for a loop  $(Q,\cdot,\backslash,/)$ to be universally elastic. Furthermore, he constructed some finite examples of loops in which this condition and the middle Bol identity $x(yz\backslash x)=(x/z)(y\backslash x)$ are equivalent, and then posed an open problem of investigating if these two identities are also equivalent in all other finite loops.

In 2012, Drapal and Shcherbacov \cite{davidref:1} reported that Kinyon constructed a non-flexible middle Bol loop of order $16$. This necessitates a reformulation of the Syrbu's open problem. Although the above authors also reported that Kinyon reformulated the Syrbu's open problem as follows: Let $Q$ be a loop such that every isotope of $Q$  is flexible and has the AAIP. Must $Q$ be a middle Bol loop? This study prepares the ground for different reformulation of Syrbu's open problem based on the fact that the algebraic properties and structural properties of middle Bol loops have been studied in the past relative to their corresponding right (left) Bol loop. Our envisioned reformulation of the equivalence of the universal elasticity condition (UEC) and the middle Bol identity (MBI) is by searching for an additional identity (AI) such that UEC = MBI+AI. In this work, we prepare a good ground to reformulate Syrbu's question:
\begin{enumerate}
\item by establishing some new algebraic properties of a middle Bol loop;
\item by investigating the relationship between a middle Bol loop and some inverse property loops like WIPLs, CIPLs, AIPLs, SAIPLs, RIPLs and IPLs;
\item  by establishing necessary and sufficient condition(s) for a middle Bol loop to be a Moufang loop or an extra loop or a group.
\end{enumerate}

\begin{mydef}\label{newdef}
Let $(Q,\cdot )$ be a loop and let $w_1(q_1,q_2,\cdots ,q_n)$ and
$w_2(q_1,q_2,\cdots ,q_n)$ be words in terms of variables
$q_1,q_2,\cdots ,q_n$ of the loop $Q$ with equal lengths $N$($N\in
\mathbb{N}$,~$N>1$) such that the variables $q_1,q_2,\cdots ,q_n$
appear in them in equal number of times. $Q$ is called a {\Large
N}$_{w_1(r_1,r_2,\cdots , r_n)=w_2(r_1,r_2,\cdots
,r_n)}^{m_1,m_2,\cdots ,m_n}$ loop if it obeys the identity
$w_1(q_1,q_2,\cdots ,q_n)=w_2(q_1,q_2,\cdots ,q_n)$ where
$m_1,m_2,\cdots ,m_n\in \mathbb{N}$ represent the number of times
the variables $q_1,q_2,\cdots ,q_n\in Q$ respectively appear in the
word $w_1$ or $w_2$ such that the mappings $q_1\mapsto
r_1,q_2\mapsto r_2,\cdots ,q_n\mapsto r_n$ are assumed,
$r_1,r_2,\cdots r_n\in \mathbb{N}$.
\end{mydef}
\begin{myrem}
The notation in Definition~\ref{newdef} was used in the study of the universality of Osborn loops in Jaiy\'e\d ol\'a and Ad\'en\'iran \cite{osborn1} when {\Large N=4}. In Phillips \cite{davidrefphi}, the case when {\Large N=5} surfaced in the characterization of WIP power-associative conjugacy closed loops with the two identities: LWPC-$(xy\cdot x)(xz)=x((yx\cdot x)z)$ and RWPC-$(zx)(x\cdot yx)=(z(x\cdot xy))x$. Kinyon et. al. \cite{davidrefkin} introduced two classes of loops that generalize Moufang and Steiner loop, namely:
\begin{itemize}
  \item RIF loop-this is an IPL that obeys the identity $(xy)(z\cdot xy)=(x\cdot yz)x\cdot y$; and
  \item WRIF loop-this is  a flexible loop that satisfies the identities $W_1:~(zx)(yxy)=z(xyx)\cdot y$ and $W_2:~(yxy)(xz)=y\cdot (xyx)z$.
\end{itemize}
 They showed that a WRIF loop is a dissociative loop and a RIF is a WRIF loop. It clear that these two loops are described by identities of the type {\Large N=5}.
\end{myrem}
\paragraph{}
In 1968, Fenyves \cite{1} obtained sixty identities of Bol-Moufang
type. These identities have four variables on each side of the
equations in the same order with one element repeating itself. Fenyves \cite{1}, Kinyon and Kunen \cite{44}, and Phillips and
Vojt\v echovsk\'y \cite{38} found some of these identities to be
equivalent to associativity in (loops) and others to
describe extra, Bol, Moufang, central, flexible loops. Some of these
sixty identities are given below following the labelling in Fenyves
\cite{1}.
\begin{multicols}{2}
\begin{description}
\item[$F_{1}$:] $xy\cdot zx=(xy\cdot z)x$ (Associativity)

\item[$F_{2}$:] $xy\cdot zx=(x\cdot yz)x$ (Moufang identity)

\item[$F_{3}$:] $xy\cdot zx=x(y\cdot zx)$ (Associativity)

\item[$F_{4}$:] $xy\cdot zx=x(yz\cdot x)$ (Moufang identity)

\item[$F_{11}$:] $xy\cdot xz=(xy\cdot x)z$ (Associativity)

\item[$F_{12}$:] $xy\cdot xz=(x\cdot yx)z$ (Associativity)

\item[$F_{13}$:] $xy\cdot xz=x(yx\cdot z)$ (extra identity)

\item[$F_{14}$:] $xy\cdot xz=x(y\cdot xz)$ (Associativity)

\item[$F_{21}$:] $yx\cdot zx=(yx\cdot z)x$ (Associativity)

\item[$F_{22}$:] $yx\cdot zx=(y\cdot xz)x$ (extra identity)

\item[$F_{23}$:] $yx\cdot zx=y(xz\cdot x)$ (Associativity)

\item[$F_{24}$:] $yx\cdot zx=y(x\cdot zx)$ (Associativity)

\item[$F_{31}$:] $yx\cdot xz=(yx\cdot x)z$ (Associativity)

\item[$F_{32}$:] $yx\cdot xz=(y\cdot xx)z$ (Associativity)

\item[$F_{33}$:] $yx\cdot xz=y(xx\cdot z)$ (Associativity)

\item[$F_{34}$:] $yx\cdot xz=y(x\cdot xz)$ (Associativity)
\end{description}
\end{multicols}

\section{Main Results}
\begin{mylem}\label{davidlem9}
Let $(Q,\cdot,\backslash,/)$ be a middle Bol loop. Then
\begin{description}
  \item[(a)] $(yz)^\rho = z^\lambda \cdot y^\rho$ and $z^\rho = z^\lambda $ i.e. $(Q,\cdot)$ is an AAIPL.
  \item[(b)] $yx\backslash x = x\backslash(y\backslash x)$.
  \item[(c)] $(yx)u = x \Leftrightarrow y(xu) = x$ and $R_uL_y =I \Leftrightarrow I = L_yR_u$.
  \item[(d)]  $xz\backslash x = x\backslash(x/z)$.
  \item[(e)] $(xz)u = x \Leftrightarrow (xu)z = x$ and $R_zR_u = I \Leftrightarrow R_uR_z = I$.
  \item[(f)] $x(z\backslash x)= (x/z)x$.
  \item[(g)] $xx = (x/z)(z^\lambda\backslash x), xx = (x/ y^\rho)(y\backslash x)$.
  \item[(h)] $|x| = 2 \Leftrightarrow (x/z)^{-1} = z^{-1} \backslash x$. Hence, $(Q,/)\equiv(Q, (\backslash)^*)$.
   \item[(i)] $(x/yz)x = (x/z)(y\backslash x)$.
  \item[(j)] $(Q,\cdot)$ is a CIPL iff $(Q,\cdot)$ is a commutative WIPL iff $(Q,\cdot)$ is a commutative IPL iff $(Q,\cdot)$ is commutative LIPL iff $(Q,\cdot)$ is commutative RIPL. Hence, $(Q,\cdot)$ is a commutative Moufang loop.
  \item[(k)] $(Q,\cdot)$ is a SAIPL iff $(Q,\cdot)$ is flexible.
  \item[(l)] $(Q,\cdot)$ is a AIPL iff $(Q,\cdot)$ is commutative. Hence, $(Q,\cdot)$ is an isostroph of a Bruck loop.
 \item[(m)] The following are equivalent:
 \begin{multicols}{3}
  \begin{enumerate}
   \item $(Q,/)\equiv(Q, \backslash)$.
   \item $x(yx\backslash x)=y(yx\backslash y)$.
   \item $(x/yx)x=y(yx\backslash y)$.
   \item $x(yx\backslash x)=(y/yx)y$.
   \item $(x/yx)x=(y/yx)y$.
 \end{enumerate}
 \end{multicols}
  \end{description}
\end{mylem}
\begin{proof}
\begin{description}
\item[(a)] Since $(Q,\cdot,\backslash,/)$ is a middle Bol loop, then
 \begin{equation}\label{davideq13.1}
x(yz\backslash x) = (x/z)(y\backslash x).
\end{equation}
Let $x=e$, then, $e(yz\backslash e)= (e/z)(y\backslash e)$.
Let $yz\backslash e = u$, then $e = yz\cdot u\Longrightarrow u=(yz)^\rho$. Let $e/z=v$, then $e=v\cdot z \Rightarrow$ $v = z^\lambda$ and let $y\backslash e= w$, then $e= y\cdot w\Longrightarrow w=y^\rho$. So $(yz\backslash e)=(e/z)(y\backslash e),\Longrightarrow (yz)^\rho=z^\lambda\cdot y^\rho$.
Let $y = e$, then $(ez)^\rho = z^\lambda \cdot e^\rho$ implies $z^\rho = z^\lambda.$
\item[(b)] Put $z = x$ in \eqref{davideq13.1}, then $x(yx\backslash x) = (x/x)(y\backslash x)= e(y\backslash x)\Rightarrow x(yx\backslash x) = y\backslash x$. Thus, $(yx)\backslash x = x\backslash(y\backslash x).$
\item[(c)] From (b), let $u = (yx)\backslash x \Leftrightarrow (yx)\cdot u = x\Leftrightarrow (yx)u = x.$
Let $x\backslash (y\backslash x) = u\Leftrightarrow y\backslash x = xu\Leftrightarrow x = y(x\cdot u).$
  Then, $y(xu)=x \Leftrightarrow  R_uL_y = I.$ Also, $(yx)u=x \Leftrightarrow L_yR_u = I.$
Therefore, $(yx)u = x\Leftrightarrow y(xu) = x$ and  $  R_uL_y = I \Leftrightarrow L_yR_u = I.$
\item[(d)] Put $y=x$ in \eqref{davideq13.1}, then $x(xz\backslash x)= (x/z)(x\backslash x)= (x/z)e$ and $x(xz\backslash x)= x/z.$
 Therefore, $ xz\backslash x=x\backslash (x/z)$.
 \item[(e)]From (d), let $u= xz\backslash x \Leftrightarrow (xz)u= x$ and let $u= x\backslash(x/z)\Leftrightarrow xu = x/z \Leftrightarrow (xu)z = x,$
that is $(xz)u = x \Leftrightarrow (xu)z = x.$
Then, $R_zR_u =I \Leftrightarrow R_uR_z = I.$
Therefore, $R_zR_u=I\Leftrightarrow R_uR_z=I.$
\item[(f)] Put $y=e$ in \eqref{davideq13.1}, then $x(ez\backslash x)=(x/z)(e\backslash x)\Rightarrow x(z\backslash x)=(x/z)x.$
\item[(g)] Put $y=z^\lambda$ in \eqref{davideq13.1}, then $x(z^\lambda z\backslash x)=(x/z)(z^\lambda\backslash x)\Rightarrow x(e\backslash x)=(x/z)(z^\lambda\backslash x).$
\item[(h)] Also, put $z=y^\rho$ in \eqref{davideq13.1}, then $x(yy^\rho\backslash x)=(x/y^\rho)(y\backslash x)\Rightarrow x(e\backslash x)=(x/y^\rho)(y\backslash x)$. This implies that $xx=(x/y^\rho)(y\backslash x).$ So, $(x/z)(z^\lambda\backslash x)=(x/y^\rho)(y\backslash x).$
\item[(i)]Assuming $x^2=e$, then $e=(x/y^\rho)(y\backslash x)\Rightarrow e/(y\backslash x)=x/y^\rho\Rightarrow (y\backslash x)^\lambda= x/y^\rho$, implies $(y\backslash x)^{-1}=x/y^{-1}$. Likewise, assuming that $x^2=e$, then $e=(x/z)(z^\lambda\backslash x)\Rightarrow (x/z)\backslash e=(z^\lambda\backslash x)$, therefore, $(x/z)^\rho= z^\lambda\backslash x$, that is, $(x/z)^{-1}=z^{-1}\backslash x \Rightarrow x/z=z\backslash x.$
$|x|=2\Leftrightarrow (x/z)^{-1}=z^{-1}\backslash x$ and $x/z=z\backslash x\Leftrightarrow x/z=x(\backslash)^*z\Rightarrow (Q,/)\equiv(Q,(\backslash)^*).$  \item[(j)] Again, in a MBL, CIP$\Leftrightarrow$ WIP and AIP $\Leftrightarrow$ IP and $(xy)^{-1}=x^{-1}y^{-1}\Leftrightarrow$ IP and $y^{-1}x^{-1}=x^{-1}y^{-1}\Leftrightarrow $ IP and $(y^{-1})^{-1}(x^{-1})^{-1}=(x^{-1})^{-1}(y^{-1})^{-1}\Leftrightarrow$ IP and $yx=xy\Leftrightarrow $IP and commutativity.
\item[(k)] SAIPL $\Leftrightarrow ((xy)\cdot x)^\rho= (x^\rho\cdot y^\rho)x^\rho\Leftrightarrow x^{-1}(xy)^{-1}=(x^{-1}\cdot y^{-1})x^{-1}\Leftrightarrow x^{-1}\cdot y^{-1}x^{-1}= (x^{-1}\cdot y^{-1})x^{-1}\Leftrightarrow (x^{-1})^{-1}\cdot (y^{-1})^{-1}(x^{-1})^{-1}=((x^{-1})^{-1}\cdot (y^{-1})^{-1})(x^{-1})^{-1}\Leftrightarrow x\cdot yx=(x\cdot y)x\Leftrightarrow$ flexibility.
\item[(l)] $xy=yx\Leftrightarrow (xy)^{-1}=(yx)^{-1}\Leftrightarrow (xy)^{-1}=x^{-1}y^{-1}\Leftrightarrow$ AIPL.
\item[(m)] From (b) and (d), $y\backslash x=x(yx\backslash x)$ and $x/z=x(xz\backslash x)$. Thus, $(Q,/)\equiv(Q, \backslash)$ if and only if $x(yx\backslash x)=y(yx\backslash y)$. The equivalence to the others follow from (f).
\end{description}
\end{proof}

\begin{mylem}\label{davidlem10}
Let $(Q,\cdot, \backslash,/)$ be a middle Bol loop. Let $f_1,g_1:Q^2\rightarrow Q$. Then:
\begin{enumerate}
  \item $f_1(x,y)=yx\backslash x\Leftrightarrow f_1(x,y)=x\backslash(y\backslash x)$;
    \begin{multicols}{3}
      \begin{enumerate}
    \item $f_1(x,e)=e$.
    \item $f_1(x^{-1},e)= e$.
    \item $f_1(e,e)= e$.
    \item $f_1(e,x)=x^{-1}$.
    \item $f_1(x,x)=x^{-1}$.
    \item $f_1(x^{-1}, x)= x^{-1}$.
    \item $f_1(e,x^{-1})=x$.
    \item $f_1(x, x^{-1})= x$.
  \end{enumerate}
  \end{multicols}
   \item $g_1(x,y)=xy\backslash x\Leftrightarrow g_1(x,y)=x\backslash (x/y)$;
  \begin{multicols}{3}
  \begin{enumerate}
    \item $g_1(x,e)=e$.
    \item $g_1(x^{-1},e)= e$.
    \item $g_1(e,e)= e$.
    \item $g_1(e,x)=x^{-1}$.
    \item $g_1(x,x)=x^{-1}$.
    \item $g_1(x^{-1}, x)= x^{-1}$.
    \item $g_1(e,x^{-1})=x$.
    \item $g_1(x, x^{-1})= x$.
  \end{enumerate}
  \end{multicols}
   \item $f_1(x,y)=g_1(x,y)\Leftrightarrow (Q,\cdot)$ is commutative.
  \item $f_1(x,y)=g_1(x,y)\Leftrightarrow (Q,(\backslash)^*)\equiv(Q,/)\Leftrightarrow (Q,\backslash)\equiv(Q,(/)^*)$.
  \item $f_1(x,y)=g_1(x,y)\Leftrightarrow yx\backslash x=x\backslash (x/y)\Leftrightarrow xy\backslash x=x\backslash (y\backslash x)$.
 \item $x=y\cdot (x/y)\Leftrightarrow(y\backslash x)\cdot y=x$
\end{enumerate}
\end{mylem}
\begin{proof}
\begin{enumerate}
  \item From Lemma~\ref{davidlem9}(b), $yx\backslash x= x\backslash(y\backslash x)$. So, $f_1(x,y)=yx\backslash x\Leftrightarrow f_1(x,y)=x\backslash(y\backslash x)$.
   \item From Lemma~\ref{davidlem9}(d), $xz\backslash x = x\backslash(x/z)$. So, $g_1(x,z)=xz\backslash x\Leftrightarrow g_1(x,z)=x\backslash(x/z)$.
     \item Since $f_1(x,y)=yx\backslash x\Leftrightarrow f_1(x,y)=x\backslash (y\backslash x)$ and $g_1(x,y)=xy\backslash x\Leftrightarrow g_1(x,y)=x\backslash (x/y)$,
      then, $f_1(x,y)=g_1(x,y)\Leftrightarrow yx\backslash x=xy\backslash x\Leftrightarrow yx=xy\Leftrightarrow (Q,\cdot)$ is commutative.
     \item $f_1(x,y)=g_1(x,y)\Leftrightarrow x\backslash (y\backslash x)=x\backslash(x/y)\Leftrightarrow y\backslash x=x/y\Leftrightarrow y\backslash x=x/y\Leftrightarrow y\backslash x=y(/)^* x\iff x(\backslash)^*y=x/y\Leftrightarrow (Q,(\backslash)^*)\equiv (Q,/)\Leftrightarrow (Q,\backslash)\equiv(Q,(/)^*).$
       \item  $f_1(x,y)=g_1(x,y)\Leftrightarrow xy\backslash x=x\backslash(x/y)$ and $yx\backslash x=x\backslash(y\backslash x).$ By equating the LHS, we have $xy\backslash x=yx\backslash x\iff x\backslash (x/y)=x\backslash(y\backslash x)\iff x/y=y\backslash x$. $$\therefore f_1(x,y)=g_1(x,y)\iff xy\backslash x=x\backslash(x/y)\iff yx\backslash x=x\backslash(y\backslash x).$$
      \item Therefore, $y\backslash x=x/y\Leftrightarrow x=y\cdot (x/y)$ or $(y\backslash x)\cdot y=x$.
  \end{enumerate}
\end{proof}

\begin{myth}\label{davidlem10.1}
Let $(Q,\cdot,\backslash,/)$ be a middle Bol loop and let $f_1,g_1:Q^2\rightarrow Q$ and $\alpha_i,\beta_i:Q^i\rightarrow Q$ be defined as:
\begin{gather*}
f_1(x,y)=yx\backslash x~\textrm{or}~f_1(x,y)=x\backslash(y\backslash x)~\textrm{and}~g_1(x,y)=xy\backslash x~\textrm{or}~ g_1(x,y)=x\backslash (x/y),\\
\alpha_i(x_1,x_2,\ldots,x_i)=(\ldots(((x_1x_2)x_3)x_4)\ldots x_{i-1})x_i~\textrm{and}\\
\beta_i(x_1,x_2,\ldots,x_i)=x_1\backslash (x_2\backslash (x_3\backslash (\cdots x_{i-2}\backslash (x_{i-1}\backslash x_i)\cdots)))~\forall~i\in\mathbb{N}.
\end{gather*}
The following are true.
\begin{enumerate}
  \item $f_1\big(x,\alpha_n(y,x,x,\ldots ,x)\big)=\beta_n\big(x,x,\ldots,x,f_1(x,y)\big)$.
  \item $f_1\big(x,\alpha_{n+1}(x,y,x,x,\ldots ,x)\big)=\beta_{n+1}\big(x,x,x,\ldots,x,g_1(x,y)\big)$.
  \item $(Q,\cdot)$ has the RAP if and only if $f_1(x,y)=x[(yx^2)\backslash x]$.
  \item $(Q,\cdot)$ has the PRAP if and only if $yx^n\cdot\beta_n\big(x,x,\ldots,x,f_1(x,y)\big)=x$.
   \item If $(Q,\cdot)$ has the RAP, then $(Q,\cdot)$ is of exponent $2$ if and only if $f_1(x,y)=x(y\backslash x)$.
   \item If $(Q,\cdot)$ has the PRAP, then $(Q,\cdot)$ is of exponent $n$ if and only if

   $y\cdot\beta_n\big(x,x,\ldots,x,f_1(x,y)\big)=x$.
\end{enumerate}
\end{myth}
\begin{proof}
\begin{enumerate}
  \item By Lemma~\ref{davidlem9}(b),
$yx\backslash x=x\backslash(y\backslash x)\Rightarrow (yx)\mathcal{R}_x=(y\backslash x)\mathcal{L}_x\Rightarrow R_x\mathcal{R}_x=\mathcal{R}_x\mathcal{L}_x\Rightarrow $
\begin{equation}\label{davideq12.1}
R_x=\mathcal{R}_x\mathcal{L}_x\mathcal{R}^{-1}_x
\end{equation}
By equation~\eqref{davideq12.1}
$$R^2_x= R_xR_x= \mathcal{R}_x\mathcal{L}_x\mathcal{R}^{-1}_x\mathcal{R}_x\mathcal{L}_x\mathcal{R}^{-1}_x
=\mathcal{R}_x\mathcal{L}^2_x\mathcal{R}^{-1}_x,$$
$$R^3_x={R^2}_xR_x= \mathcal{R}_x\mathcal{L}^2_x\mathcal{R}^{-1}_x\mathcal{R}_x\mathcal{L}_x\mathcal{R}^{-1}_x=\mathcal{R}_x\mathcal{L}^3_x\mathcal{R}^{-1}_x,$$
$$R^4_x=\mathcal{R}_x\mathcal{L}^3_x\mathcal{R}^{-1} \mathcal{R}_x\mathcal{L}_x\mathcal{R}^{-1}_x=\mathcal{R}_x\mathcal{L}^4_x\mathcal{R}^{-1}_x.$$
Therefore, we claim that: $R^n_x=\mathcal{R}_x\mathcal{L}^{n-1}_x\mathcal{R}^{-1}_x \mathcal{R}_x\mathcal{L}_x\mathcal{R}^{-1}_x=\mathcal{R}_x\mathcal{L}^n_x\mathcal{R}^{-1}_x, ~n\geq 0$. Thus, for all $y\in Q$,
\begin{equation}\label{davideq17af}
(\cdots((y\underbrace{x\cdot x)x\cdot x)x\cdots)x}_{\textrm{$n$-times}}\backslash x=
(\underbrace{x\backslash\cdots(x\backslash(x}_{\textrm{$(n-1)$-times}}\backslash(y\backslash x)))\cdots)
\end{equation}
Equation~\eqref{davideq17af} implies that $f_1\big(x,\alpha_n(y,x,x,\ldots ,x)\big)=\beta_n\big(x,x,\ldots,x,f_1(x,y)\big)$.
  \item By Lemma~\ref{davidlem9}(d), $xz\backslash x=x\backslash(x/z)\Rightarrow (xz)\mathcal{R}_x=(x/z)\mathcal{L}_x\Rightarrow zL_x\mathcal{R}_x=z\mathbb{L}_x\mathcal{L}_x\Rightarrow L_x\mathcal{R}_x=\mathbb{L}_x\mathcal{L}_x\Rightarrow$
\begin{equation}\label{davideq13}
L_x=\mathbb{L}_x\mathcal{L}_x\mathcal{R}^{-1}_x
\end{equation}
By equation~\eqref{davideq12.1} and equation~\eqref{davideq13},
$$L_xR_x=\mathbb{L}_x\mathcal{L}_x\mathcal{R}^{-1}_x\mathcal{R}_x\mathcal{L}_x\mathcal{R}^{-1}_x=
\mathbb{L}_x\mathcal{L}_x\mathcal{L}_x\mathcal{R}^{-1}_x=\mathbb{L}_x\mathcal{L}^2_x\mathcal{R}^{-1}_x.$$
$$L_xR^2_x=\mathbb{L}_x\mathcal{L}^2_x\mathcal{R}^{-1}_x\mathcal{R}_x\mathcal{L}_x\mathcal{R}^{-1}_x=\mathbb{L}_x\mathcal{L}^3_x\mathcal{R}^{-1}_x.$$
$$L_xR^3_x=\mathbb{L}_x\mathcal{L}^3_x\mathcal{R}^{-1}_x\mathcal{R}_x\mathcal{L}_x\mathcal{R}^{-1}_x=\mathbb{L}_x\mathcal{L}^4_x\mathcal{R}^{-1}_x.$$
Therefore, $L_xR^n_x= \mathbb{L}_x\mathcal{L}^{(n+1)}_x\mathcal{R}^{-1}_x, ~n\geq 1$. Thus, for all $y\in Q$,
\begin{equation}\label{davideq18af}
(\cdots((xy\cdot\underbrace{ x)x\cdot x)x\cdots)x}_{\textrm{$n$-times}}\backslash x=
(\underbrace{x\backslash\cdots(x\backslash(x}_{\textrm{$(n+1)$-times}}\backslash(x/y)))\cdots)
\end{equation}
Equation~\eqref{davideq18af} implies that $f_1\big(x,\alpha_{n+1}(x,y,x,x,\ldots ,x)\big)=\beta_{n+1}\big(x,x,x,\ldots,x,g_1(x,y)\big)$.
  \item This follows from 1. when $n=2$.
  \item This follows from 1.
  \item This follows from 3.
  \item This follows from 4.
\end{enumerate}
\end{proof}

\begin{mylem}\label{davidlem11}
Let $(Q,\cdot,\backslash,/)$ be a loop. The following are equivalent.
\begin{enumerate}
  \item $(Q,\cdot,\backslash,/)$ be a middle Bol loop.
  \item $x(yz\backslash x) = (x/z)(y\backslash x)$ for all $x,y,z\in Q$.
  \item $(x/yz)x = (x/z)(y\backslash x)$ for all $x,y,z\in Q$.
\end{enumerate}
\end{mylem}
\begin{proof}
From Lemma~\ref{davidlem9}(f), $x(z\backslash x)= (x/z)x$. On another hand, if $(x/yz)x = (x/z)(y\backslash x)$ is true, then $x(y\backslash x)= (x/y)x$.
So, 1., 2. and 3. are equivalent.
\end{proof}

\begin{myth}\label{davidlem12}
Let $(Q,\cdot,\backslash,/)$ be a middle Bol loop and let $f_1,g_1,f_2,g_2:Q^2\rightarrow Q$ be defined as:
\begin{gather*}
f_1(x,y)=yx\backslash x~\textrm{or}~f_1(x,y)=x\backslash(y\backslash x)~\textrm{and}~g_1(x,y)=xy\backslash x~\textrm{or}~ g_1(x,y)=x\backslash (x/y),\\
f_2(x,y)=x/(xy)~\textrm{or}~f_2(x,y)=(x/y)/x~\textrm{and}~g_2(x,y)=x/(yx)~\textrm{or}~ g_2(x,y)=(y\backslash x)/x.
\end{gather*}
Then:
\begin{description}
  \item[(a)] $x/yx=(y\backslash x)/x$.
  \item[(b)] $z(yx) = x \Leftrightarrow y(zx) = x$ and $L_yL_z = I \Leftrightarrow L_zL_y = I$.
  \item[(c)] $x/(xz) = (x/z)/x$.
    \item[(d)] $(yx)u = x \Leftrightarrow y(xu) = x$ and $R_uL_y =I \Leftrightarrow I = L_yR_u$.
  \item[(e)] $f_2(x,y)=x/(xy)\Leftrightarrow f_2(x,y)=(x/y)/x$.
  \item[(f)]  $g_2(x,y)=x/(yx)\Leftrightarrow g_2(x,y)=(y\backslash x)/x$.
     \item[(g)]  The following are equivalent:
     \begin{multicols}{3}
  \begin{enumerate}
   \item $(Q,/)\equiv(Q, \backslash)$.
   \item $[x/(xy)]x=[y/(xy)]y$.
   \item $x[(xy)\backslash x]=[y/(xy)]y$.
   \item $[x/(xy)]x=y[(xy)\backslash y]$.
   \item $x[(xy)\backslash x]=y[(xy)\backslash y]$.
 \end{enumerate}
 \end{multicols}
       \item[(i)] $yx\cdot z =x\Leftrightarrow xz = [x/(yx)]x\Leftrightarrow y\cdot xz =x$.
  \item[(j)] $(Q,\cdot )$ is a CIPL if and only if $xy^{-1}= [x/(yx)]x$.
  \item[(k)] $yx\cdot z =x\Leftrightarrow xz = g_2(x,y)\cdot x\Leftrightarrow y\cdot xz =x$.
\item[(l)] $(Q,\cdot )$ is a CIPL if and only if $xy^{-1}= g_2(x,y)\cdot x$.
 \item[(m)] $z\cdot xy =x\Leftrightarrow zx = x[(xy)\backslash x]\Leftrightarrow zx\cdot y =x$.
\item[(n)] $(Q,\cdot )$ is a CIPL if and only if $y^{-1}x= x[(xy)\backslash x]$.
\item[(o)] $z\cdot xy =x\Leftrightarrow zx = x\cdot g_1(x,y)$.
\item[(p)] $(Q,\cdot )$ is a CIPL if and only if $y^{-1}x= x\cdot g_1(x,y)$.
\item[(q)]  $z\cdot yx =x\Leftrightarrow zx = x[(yx)\backslash x]\Leftrightarrow y\cdot zx =x$.
\item[(r)] $(Q,\cdot )$ is a LIPL if and only if $y^{-1}x= x[(yx)\backslash x]$.
\item[(s)] $z\cdot yx =x\Leftrightarrow zx = x\cdot f_1(x,y)$.
\item[(t)] $(Q,\cdot )$ is a LIPL if and only if $y^{-1}x= x\cdot f_1(x,y)$.
\item[(u)] $xy\cdot z =x\Leftrightarrow xz = [x/(xy)]x$.
\item[(v)] $(Q,\cdot )$ is a RIPL if and only if $xy^{-1}= [x/(xy)]x$.
\item[(w)] $xy\cdot z =x\Leftrightarrow xz = f_2(x,y)\cdot x$.
\item[(x)] $(Q,\cdot )$ is a RIPL if and only if $xy^{-1}= f_2(x,y)\cdot x$.
\end{description}
\end{myth}
\begin{proof}
This is achieved by using the identity in 3. of Lemma~\ref{davidlem11} i.e.
\begin{equation}\label{MBLI2}
  (x/yz)x = (x/z)(y\backslash x)
\end{equation}
the ways in which the identity in 2. of Lemma~\ref{davidlem11} was used to prove the results in Lemma~\ref{davidlem9}.
\begin{description}
  \item[(a)] Substitute $z=x$ in \eqref{MBLI2}.
    \item[(b)] Use (a).
  \item[(c)] Substitute $y=x$ in \eqref{MBLI2}.
    \item[(d)] Use (c).
  \item[(e)] Follows from (c).
  \item[(f)]  Follows from (a).
   \item[(g)] From (a) and (c), $x\backslash y=(y/xy)y $ and $x/z=(x/xz)x$. So, $(Q,\backslash)\equiv(Q, /)\Leftrightarrow [x/(xy)]x=[y/(xy)]y$.
  The equivalence to the others follows from Lemma~\ref{davidlem9}(f).
         \item[(i)] Let $z=yx\backslash x$, then $yx\cdot z =x$. So, $xz = [x/(yx)]x$. Using Lemma~\ref{davidlem9}(c) in addition,
  $yx\cdot z =x\Leftrightarrow xz = [x/(yx)]x\Leftrightarrow y\cdot xz =x$.
  \item[(j)] Apply (i).
  \item[(k)] Use (i).
\item[(l)] Apply (k).
 \item[(m)] Let $z=x/xy$, then $z\cdot xy =x$. So, $zx = x[(xy)\backslash x]$. Using Lemma~\ref{davidlem9}(c) in addition,
 $z\cdot xy =x\Leftrightarrow zx = x[(xy)\backslash x]\Leftrightarrow zx\cdot y =x$.
 \item[(n)] Apply (m).
\item[(o)] Use (m).
\item[(p)] Apply (o).
\item[(q)] Let $z=x/yx$, then $z\cdot yx =x$. So, $zx = x[(yx)\backslash x]$. Using (b) in addition,
 $z\cdot yx =x\Leftrightarrow zx = x[(yx)\backslash x]\Leftrightarrow y\cdot zx =x$.
\item[(r)] Apply (q).
\item[(s)] Use (q).
\item[(t)] Apply (s).
\item[(u)] Let $z=xy\backslash x$, then $xy\cdot z =x$. So, $zx = x[(yx)\backslash x]$. Using Lemma~\ref{davidlem9}(e) in addition,
 $xy\cdot z =x\Leftrightarrow xz = [x/(xy)]x\Leftrightarrow xz\cdot y =x$.
\item[(v)] Apply (u).
\item[(w)] Use (u).
\item[(x)] Apply (w)$(Q,\cdot )$ is a RIPL if and only if $xy^{-1}= f_2(x,y)\cdot x$.
\end{description}
\end{proof}

\begin{mylem}\label{davidlem13}
Let $(Q,\cdot, \backslash,/)$ be a middle Bol loop and let $f_1,g_1,f_2,g_2:Q^2\rightarrow Q$.
\begin{enumerate}
  \item $f_2(x,y)=x/(xy)\Leftrightarrow f_2(x,y)=(x/y)/x$;
      \begin{enumerate}
    \item $f_2(x,x)=f_2(e,x)=f_2(x^{-1}, x)=x^{-1}$.
     \item $f_2(x, x^{-1})=f_2(e,x^{-1})= x$.
    \item $f_2(x,e)=f_2(x^{-1},e)=f_2(e,e)= e$.
        \end{enumerate}
          \item  $g_2(x,y)=x/(yx)\Leftrightarrow g_2(x,y)=(y\backslash x)/x$;
        \begin{enumerate}
    \item $g_2(x,x)=g_2(e,x)=g_2(x^{-1},x)=x^{-1}$.
    \item $g_2(x, x^{-1})=g_2(e,x^{-1})= x$.
    \item $g_2(x,e)=g_2(x^{-1},e)=g_2(e,e)= e$.
          \end{enumerate}
           \item The following are equivalent:
      \begin{multicols}{3}
       \begin{enumerate}
     \item $f_2(x,y)=g_2(x,y)$.
     \item $(Q,\cdot)$ is commutative.
       \item $(Q,(\backslash)^*)\equiv(Q,/)$.
       \item $(Q,\backslash)\equiv(Q,(/)^*)$.
  \item $x/xy=(y\backslash x)/x$.
  \item $x/yx=(x/y)/x$.
  \end{enumerate}
  \end{multicols}
  \item The following are equivalent:
      \begin{multicols}{3}
       \begin{enumerate}
     \item $f_1(x,y)=f_2(x,y)$.
      \item $(yx\backslash x)(xy)=x$.
  \item $(yx)(x/xy)=x$.
  \item $x\backslash (y\backslash x)=(x/y)/x$.
  \end{enumerate}
  \end{multicols}
  \item The following are equivalent:
      \begin{multicols}{3}
       \begin{enumerate}
     \item $g_1(x,y)=g_2(x,y)$.
      \item $(xy\backslash x)(yx)=x$.
  \item $(xy)(x/yx)=x$.
  \item $x\backslash (x/y)=(y\backslash x)/x$.
  \end{enumerate}
  \end{multicols}
  \item The following are equivalent:
      \begin{multicols}{3}
       \begin{enumerate}
     \item $f_1(x,y)=g_2(x,y)$.
      \item $(yx)(x/yx)=x$.
  \item $(yx\backslash x)(yx)=x$.
  \item $x[(y\backslash x)/x]=y\backslash x$.
  \item $[(y\backslash x)/x]x=y\backslash x$.
  \end{enumerate}
  \end{multicols}
 \item The following are equivalent:
      \begin{multicols}{3}
       \begin{enumerate}
     \item $f_2(x,y)=g_1(x,y)$.
      \item $(xy\backslash x)(xy)=x$.
  \item $(yx\backslash x)(yx)=x$.
  \item $x[(x/y)/x]=x/y$.
  \item $[x\backslash (x/y)]x=x/y$.
  \end{enumerate}
  \end{multicols}
  \item $f_2(x,y)f_1(x,z)=x\big[(zx)(xy)\backslash x\big]=\big[x/(zx)(xy)\big]x$.
  \item $[(x/y)/x][x\backslash (z\backslash x)]=x\big[(zx)(xy)\backslash x\big]=\big[x/(zx)(xy)\big]x$.
 \item  $g_2(x,y)g_1(x,z)=x\big[(xz)(yx)\backslash x\big]=\big[x/(xz)(yx)\big]x$.
 \item  $\big[(y\backslash x)/x\big]\big[x\backslash (x/z)]=x\big[(xz)(yx)\backslash x\big]=\big[x/(xz)(yx)\big]x$.
 \item  $f_2(x,y)g_1(x,z)=x\big[(xz)(xy)\backslash x\big]=\big[x/(xz)(xy)\big]x$.
 \item  $\big[(x/y)/x\big]\big[x\backslash (x/z)]=x\big[(xz)(xy)\backslash x\big]=\big[x/(xz)(xy)\big]x$.
 \item  $g_2(x,y)f_1(x,z)=x\big[(xz)(yx)\backslash x\big]=\big[x/(xz)(yx)\big]x$.
 \item  $\big[(y\backslash x)/x\big]\big[x\backslash (z\backslash x)]=x\big[(xz)(yx)\backslash x\big]=\big[x/(xz)(yx)\big]x$.
 \end{enumerate}
\end{mylem}
\begin{proof}
Use Theorem~\ref{davidlem12} and the hypothetic definitions of $f_i,g_i,~i=1,2$.
\end{proof}

\begin{myth}\label{davidlem14}
Let $(Q,\cdot,\backslash,/)$ be a middle Bol loop and let $f_2,g_2:Q^2\rightarrow Q$ and $\phi_i,\psi_i:Q^i\rightarrow Q$ be defined as:
\begin{gather*}
f_2(x,y)=x/(xy)~\textrm{or}~f_2(x,y)=(x/y)/x~\textrm{and}~g_2(x,y)=x/(yx)~\textrm{or}~ g_2(x,y)=(y\backslash x)/x,\\
\phi_i(x_1,x_2,\ldots,x_i)=x_i(x_{i-1}(\ldots (x_5(x_4(x_3(x_2x_1))))\ldots ))~\textrm{and}\\
\psi_i(x_1,x_2,\ldots,x_i)=((\ldots((x_1/x_2)/x_3)\ldots )/x_{i-1})/x_i~\forall~i\in\mathbb{N}.
\end{gather*}
The following are true.
\begin{enumerate}
  \item $f_2\big(x,\phi_n(y,x,x,\ldots ,x)\big)=\psi_n\big(f_2(x,y),x,x,\ldots,x\big)$.
  \item $f_2\big(x,\phi_{n+1}(x,y,x,x,\ldots ,x,x)\big)=\psi_{n+1}\big(g_2(x,y),x,x,\ldots,x\big)$.
  \item $(Q,\cdot)$ has the LAP if and only if $f_2(x,y)=[x/(x^2y)]x$.
  \item $(Q,\cdot)$ has the PLAP if and only if $\psi_n\big(f_2(x,y),x,x,\ldots,x\big)\cdot x^ny=x$.
   \item If $(Q,\cdot)$ has the LAP, then $(Q,\cdot)$ is of exponent $2$ if and only if $f_2(x,y)=(x/y)x$.
   \item If $(Q,\cdot)$ has the PLAP, then $(Q,\cdot)$ is of exponent $n$ if and only if

   $\psi_n\big(f_2(x,y),x,\ldots,x\big)\cdot y=x$.
\end{enumerate}
\end{myth}
\begin{proof}
This is very much similar to the proof of Theorem~\ref{davidlem10.1}.
\begin{enumerate}
  \item From the identity in Lemma~\ref{davidlem9}(c), we get
\begin{equation}\label{davideq12.1.1}
L_x=\mathbb{L}_x\mathbb{R}_x\mathbb{L}^{-1}_x
\end{equation}
By equation~\eqref{davideq12.1.1}, we claim that:
$L^n_x=\mathbb{L}_x\mathbb{R}^{n}_x\mathbb{L}^{-1}_x ,~n\geq 0$. Thus, for all $y\in Q$,
\begin{equation}\label{davideq17af.1}
x/[\underbrace{x(\cdots x(x(x\cdot x}_{\textrm{$n$-times}}y)))\cdots]=
((((x/y)/\underbrace{x)/x)/x\cdots)/x}_{\textrm{$n$-times}}
\end{equation}
Equation~\eqref{davideq17af.1} implies that $f_2\big(x,\phi_n(y,x,x,\ldots ,x)\big)=\psi_n\big(f_2(x,y),x,x,\ldots,x\big)$.
  \item From the identity in Lemma~\ref{davidlem9}(c), we get
\begin{equation}\label{davideq13.12}
R_x=\mathcal{R}_x\mathbb{R}_x\mathbb{L}^{-1}_x
\end{equation}
Therefore, by equation~\eqref{davideq12.1.1} and equation~\eqref{davideq13.12},
$R_xL^n_x\mathbb{L}_x= \mathcal{R}_x\mathbb{R}^{(n+1)}_x, ~n\geq 1$. Thus, for all $y\in Q$,
\begin{equation}\label{davideq18af.1}
x/[\underbrace{x(\cdots x(x(x}_{\textrm{$n$-times}}\cdot yx)))\cdots]=
((((y\backslash x)/\underbrace{x)/x)/x\cdots)/x}_{\textrm{$(n+1)$-times}}
\end{equation}
Equation~\eqref{davideq18af.1} implies that $f_2\big(x,\phi_{n+1}(x,y,x,x,\ldots ,x,x)\big)=\psi_{n+1}\big(g_2(x,y),x,x,\ldots,x\big)$.
  \item This follows from 1. when $n=2$.
  \item This follows from 1.
  \item This follows from 3.
  \item This follows from 4.
\end{enumerate}
\end{proof}

\begin{myth}\label{davidlem15}
Let $(Q,\cdot,\backslash,/)$ be a middle Bol loop and let $f_1,g_1,f_2,g_2:Q^2\rightarrow Q$ and $\alpha_i,\beta_i,\phi_i,\psi_i:Q^i\rightarrow Q$ be defined as:
\begin{gather*}
f_1(x,y)=yx\backslash x~\textrm{or}~f_1(x,y)=x\backslash(y\backslash x)~\textrm{and}~g_1(x,y)=xy\backslash x~\textrm{or}~ g_1(x,y)=x\backslash (x/y),\\
f_2(x,y)=x/(xy)~\textrm{or}~f_2(x,y)=(x/y)/x~\textrm{and}~g_2(x,y)=x/(yx)~\textrm{or}~ g_2(x,y)=(y\backslash x)/x,\\
\alpha_i(x_1,x_2,\ldots,x_i)=(\ldots(((x_1x_2)x_3)x_4)\ldots x_{i-1})x_i,\\
\beta_i(x_1,x_2,\ldots,x_i)=x_1\backslash (x_2\backslash (x_3\backslash (\cdots x_{i-2}\backslash (x_{i-1}\backslash x_i)\cdots))),\\
\phi_i(x_1,x_2,\ldots,x_i)=x_i(x_{i-1}(\ldots (x_5(x_4(x_3(x_2x_1))))\ldots ))~\textrm{and}\\
\psi_i(x_1,x_2,\ldots,x_i)=((\ldots((x_1/x_2)/x_3)\ldots )/x_{i-1})/x_i~\forall~i\in\mathbb{N}.
\end{gather*}
\begin{enumerate}
  \item The following are equivalent.
  \begin{enumerate}
    \item $(Q,\cdot)$ is a group.
    \item $x/y=\big[f_2(x,y)f_1(x,z)\big]/\beta_4(x,x,z,x)$.
    \item $x/y=\big[f_2(x,y)f_1(x,z)\big]/\beta_2\big(\phi_3(x,x,z),x\big)$.
    \item $z\backslash x=\psi_2\big(x,\alpha_3(x,x,y)\big)\backslash \big[f_2(x,y)f_1(x,z)\big]$.
    \item $z\backslash x=\psi_4(x,y,x,x)\backslash \big[f_2(x,y)f_1(x,z)\big]$.
    \item $\alpha_3(x,z,y)\cdot g_2(x,y)g_1(x,z)=x$.
    \item $g_2(x,y)g_1(x,z)\cdot \phi_3(x,z,y)=x$.
    \item $f_2(x,y)g_1(x,z)=\alpha_2\big(f_2(x,y),x\big)\beta_2\big(x,g_1(x,z)\big)$.
    \item $f_2(x,y)g_1(x,z)\cdot \phi_3(x,y,z)=x$.
    \item $\alpha_3(z,x,y)\cdot g_2(x,y)f_1(x,z)=x$.
    \item $g_2(x,y)f_1(x,z)=\psi_2\big(g_2(x,y),x\big)\phi_2\big(f_1(x,z),x\big)$.
  \end{enumerate}
     \item If $(Q,\cdot)$ is of exponent $2$, then $(Q,\cdot)$ is a group if and only if $(x/y)(z\backslash x)=f_2(x,y)f_1(x,z)$.
   \item If $(Q,\cdot)$ is flexible, then $(Q,\cdot)$ is a group if and only if $f_2(x,y)g_1(x,z)=\phi_2\big(x,f_2(x,y)\big)\psi_2\big(g_2(x,z),\big)$.
    \item The following are equivalent.
    \begin{multicols}{2}
    \begin{enumerate}
      \item $(Q,\cdot)$ is a Moufang loop.
      \item $\phi_3(z,y,x)\cdot g_2(x,y)g_1(x,z)=x$.
      \item $g_2(x,y)g_1(x,z)\cdot \alpha_3(y,z,x)=x$.
    \end{enumerate}
    \end{multicols}
    \item The following are equivalent.
    \begin{multicols}{2}
    \begin{enumerate}
      \item $(Q,\cdot)$ is an extra loop.
       \item $f_2(x,y)g_1(x,z)\cdot \alpha_3(z,x,y)=x$.
      \item $\phi_3(y,x,z)\cdot g_2(x,y)f_1(x,z)=x$.
         \end{enumerate}
         \end{multicols}
\end{enumerate}
\end{myth}
 \begin{proof}
\begin{enumerate}
  \item
  \begin{description}
    \item[(a)$\Leftrightarrow$(b)] Using $F_{31}$ and 8. of Lemma~\ref{davidlem13}, $(Q,\cdot)$ is a group if and only if
        \begin{gather*}
        f_2(x,y)f_1(x,z)=x\big[(zx\cdot x)y\backslash x\big]=(x/y)\big[(zx\cdot x)\backslash x\big]=(x/y)\big[x\backslash (x\backslash (z\backslash x))\big]\Longleftrightarrow   \\
          x/y=\big[f_2(x,y)f_1(x,z)\big]/\beta_4(x,x,z,x).
        \end{gather*}
  \item[(a)$\Leftrightarrow$(c)] Using $F_{32}$ and 8. of Lemma~\ref{davidlem13}, $(Q,\cdot)$ is a group if and only if
  \begin{gather*}
        f_2(x,y)f_1(x,z)=x\big[(z\cdot xx)y\backslash x\big]=(x/y)\big[(z\cdot xx)\backslash x\big]\Longleftrightarrow   \\
          x/y=\big[f_2(x,y)f_1(x,z)\big]/\beta_2\big(\phi_3(x,x,z),x\big).
        \end{gather*}
    \item[(a)$\Leftrightarrow$(d)] Using $F_{33}$ and 8. of Lemma~\ref{davidlem13}, $(Q,\cdot)$ is a group if and only if
  \begin{gather*}
        f_2(x,y)f_1(x,z)=x\big[z(xx\cdot y)\backslash x\big]=\big(x/(xx\cdot y)\big)(z\ x)\Longleftrightarrow   \\
          z\backslash x=\psi_2\big(x,\alpha_3(x,x,y)\big)\backslash \big[f_2(x,y)f_1(x,z)\big].
        \end{gather*}
        \item[(a)$\Leftrightarrow$(e)] Using $F_{34}$ and 8. of Lemma~\ref{davidlem13}, $(Q,\cdot)$ is a group if and only if
  \begin{gather*}
        f_2(x,y)f_1(x,z)=x\big[z(x\cdot xy)\backslash x\big]=\big(x/(x\cdot xy)\big)(z\ x)=\big[\big((x/y)/x\big)/x\big](z\ x)\Longleftrightarrow   \\
          z\backslash x=\psi_4(x,y,x,x)\backslash \big[f_2(x,y)f_1(x,z)\big].
        \end{gather*}
        \item[(a)$\Leftrightarrow$(f)] Using $F_{1}$ and 10. of Lemma~\ref{davidlem13}, $(Q,\cdot)$ is a group if and only if
  \begin{gather*}
  g_2(x,y)g_1(x,z)=x\big[(xz\cdot y)x\backslash x\big]=(xz\cdot y)\backslash x\Longleftrightarrow   \\
        \alpha_3(x,z,y)\cdot g_2(x,y)g_1(x,z)=x.
        \end{gather*}
        \item[(a)$\Leftrightarrow$(g)] Using $F_{3}$ and 10. of Lemma~\ref{davidlem13}, $(Q,\cdot)$ is a group if and only if
  \begin{gather*}
  g_2(x,y)g_1(x,z)=x\big[x(y\cdot zx)\backslash x\big]=x/(y\cdot zx)\Longleftrightarrow   \\
        g_2(x,y)g_1(x,z)\cdot \phi_3(x,z,y)=x.
        \end{gather*}
        \item[(a)$\Leftrightarrow$(h)] Using $F_{11}$ and 12. of Lemma~\ref{davidlem13}, $(Q,\cdot)$ is a group if and only if
  \begin{gather*}
  f_2(x,y)g_1(x,z)=x\big[(xz\cdot x)y\backslash x\big]=(x/y)\big[(xz\cdot x)\backslash x\big]=(x/y)\big[(x\backslash (x\backslash (x/z)))\big]\Longleftrightarrow   \\
         f_2(x,y)g_1(x,z)=\big(f_2(x,y)\cdot x\big)\big(x\ g_1(x,z)\big)          \\
        f_2(x,y)g_1(x,z)=\alpha_2\big(f_2(x,y),x\big)\beta_2\big(x,g_1(x,z)\big).
        \end{gather*}
 \item[(a)$\Leftrightarrow$(i)] Using $F_{14}$ and 12. of Lemma~\ref{davidlem13}, $(Q,\cdot)$ is a group if and only if
  \begin{gather*}
  f_2(x,y)g_1(x,z)=x\big[x(zx\cdot y)\backslash x\big]=x/(z\cdot xy)\Longleftrightarrow   \\
        f_2(x,y)g_1(x,z)\cdot \phi_3(x,y,z)=x.
        \end{gather*}
         \item[(a)$\Leftrightarrow$(j)] Using $F_{21}$ and 14. of Lemma~\ref{davidlem13}, $(Q,\cdot)$ is a group if and only if
  \begin{gather*}
  g_2(x,y)f_1(x,z)=x\big[(zx\cdot y)x\backslash x\big]=(zx\cdot y\backslash x)\Longleftrightarrow   \\
       \alpha_3(z,x,y)\cdot g_2(x,y)f_1(x,z)=x.
        \end{gather*}
\item[(a)$\Leftrightarrow$(k)] Using $F_{24}$ and 14. of Lemma~\ref{davidlem13}, $(Q,\cdot)$ is a group if and only if
  \begin{gather*}
  g_2(x,y)f_1(x,z)=x\big[(z(x\cdot yx)\backslash x\big]=\big[x/(x\cdot yx)\big](z\backslash x)=\big[\big((y\backslash x)/x\big)/x\big](z\backslash x)\Longleftrightarrow   \\
   g_2(x,y)f_1(x,z)=\big(g_2(x,y)/x\big)\big(x\cdot f_1(x,z)\big)\Longleftrightarrow   \\
      g_2(x,y)f_1(x,z)=\psi_2\big(g_2(x,y),x\big)\phi_2\big(f_1(x,z),x\big).
        \end{gather*}
  \end{description}
     \item Apply 1.(c)
   \item Using $F_{23}$ and 14. of Lemma~\ref{davidlem13}, if $(Q,\cdot)$ is flexible, then $(Q,\cdot)$ is a group if and only if
  \begin{gather*}
  g_2(x,y)f_1(x,z)=x\big[(z(xy\cdot x)\backslash x\big]=\big[x/(xy\cdot x)\big](z\backslash x)=\big[x/\big(x\cdot yx)\big](z\backslash x)\Longleftrightarrow\\
 g_2(x,y)f_1(x,z)=\big[\big((y\backslash x)/x\big)/x\big](z\backslash x) \Longleftrightarrow   \\
   g_2(x,y)f_1(x,z)=\big(g_2(x,y)/x\big)\big(x\cdot f_1(x,z)\big)\Longleftrightarrow   \\
      g_2(x,y)f_1(x,z)=\psi_2\big(g_2(x,y),x\big)\phi_2\big(f_1(x,z),x\big).
        \end{gather*}
    \item
        \begin{description}
      \item[(a)$\Leftrightarrow$(b)] Using $F_{2}$ and 10. of Lemma~\ref{davidlem13}, $(Q,\cdot)$ is a Moufang loop if and only if
  \begin{gather*}
  g_2(x,y)g_1(x,z)=x\big[(x\cdot yz)x\backslash x\big]=(x\cdot yz)\backslash x\Longleftrightarrow\\
 \phi_3(z,y,x)\cdot g_2(x,y)g_1(x,z)=x.
           \end{gather*}
    \item[(a)$\Leftrightarrow$(c)] Using $F_{4}$ and 10. of Lemma~\ref{davidlem13}, $(Q,\cdot)$ is a Moufang loop if and only if
  \begin{gather*}
  g_2(x,y)g_1(x,z)=x\big[x(yz\cdot x)\backslash x\big]=x/(yz\cdot x)\Longleftrightarrow\\
 g_2(x,y)g_1(x,z)\cdot \alpha_3(y,z,x) =x.
 \end{gather*}
          \end{description}
        \item
         \begin{description}
      \item[(a)$\Leftrightarrow$(b)] Using $F_{13}$ and 12. of Lemma~\ref{davidlem13}, $(Q,\cdot)$ is an extra loop if and only if
  \begin{gather*}
  f_2(x,y)g_1(x,z)=x\big[x(zx\cdot y)\backslash x\big]=x/(zx\cdot y)\Longleftrightarrow\\
 f_2(x,y)g_1(x,z)\cdot \alpha_3(z,x,y)=x.
           \end{gather*}
    \item[(a)$\Leftrightarrow$(c)] Using $F_{22}$ and 14. of Lemma~\ref{davidlem13}, $(Q,\cdot)$ is an extra loop if and only if
  \begin{gather*}
  g_2(x,y)f_1(x,z)=x\big[(z\cdot xy)x\backslash x\big]=(z\cdot xy)\Longleftrightarrow\\
 \phi_3(y,x,z)\cdot g_2(x,y)f_1(x,z)=x.
           \end{gather*}
          \end{description}
\end{enumerate}
\end{proof}
We shall now establish some necessary and sufficient conditions for some identities of the type {\Large N=3,4,5,6} to be true in a middle loop. Of course, these identities are obviously true in a dissociative loop, hence, a RIF or WRIF loop has them.

\begin{myth}\label{davidlem16}
Let $(Q,\cdot,\backslash,/)$ be a middle Bol loop and let $g_1:Q^2\rightarrow Q$ be defined as:
$g_1(x,y)=xy\backslash x~\textrm{or}~ g_1(x,y)=x\backslash (x/y)$.
\begin{enumerate}
  \item $(Q,\cdot)$ is a {\Large 3}$_{12\cdot 1=1\cdot 21}^{2,1}$ loop if and only if $g_1(x,y)=x\cdot g_1(x,yx)$.
  \item $(Q,\cdot)$ is a {\Large 4}$_{(12\cdot 1)1=1\cdot (21\cdot 1)}^{3,1}$ loop if and only if $g_1(x,y)=x\cdot \big(x\cdot g_1(x,yx\cdot x)\big)$.
  \item $(Q,\cdot)$ is a {\Large 4}$_{(12\cdot 1)1=1\cdot (2\cdot 11)}^{3,1}$ loop if and only if $g_1(x,y)=x\cdot \big(x\cdot g_1(x,y\cdot xx)\big)$.
  \item If $(Q,\cdot)$ is a {\Large 3}$_{12\cdot 2=1\cdot 22}^{1,2}$ loop, the following are equivalent:
  \begin{multicols}{2}
  \begin{enumerate}
    \item $(Q,\cdot)$ is a {\Large 4}$_{(12\cdot 1)1=1\cdot (21\cdot 1)}^{3,1}$ loop.
    \item $(Q,\cdot)$ is a {\Large 4}$_{(12\cdot 1)1=1\cdot (2\cdot 11)}^{3,1}$ loop.
    \item $(Q,\cdot)$ is a {\Large 4}$_{12\cdot 11=1\cdot (2\cdot 11)}^{3,1}$ loop.
  \end{enumerate}
  \end{multicols}
\end{enumerate}
\end{myth}
\begin{proof}
We shall often use Equation~\eqref{davideq18af}.
\begin{enumerate}
  \item $(Q,\cdot)$ is a {\Large 3}$_{12\cdot 1=1\cdot 21}^{2,1}$ loop if and only if $xy\cdot x=x\cdot yx$
  \begin{gather*}
    \Leftrightarrow (xy\cdot x)\backslash x =(x\cdot yx)\backslash x\Leftrightarrow x\backslash (x\backslash (x/y))= (x\cdot yx)\backslash x\Leftrightarrow  g_1(x,y)=x\cdot g_1(x,yx).
      \end{gather*}
    \item $(Q,\cdot)$ is a {\Large 4}$_{(12\cdot 1)1=1\cdot (21\cdot 1)}^{3,1}$ loop if and only if $(xy\cdot x)x=x\cdot (yx\cdot x)$
    \begin{gather*}
    \Leftrightarrow [(xy\cdot x)x]\backslash x = [x\cdot (yx\cdot x)]\backslash x\Leftrightarrow x\backslash (x\backslash (x\backslash (x/y)))=[x\cdot (yx\cdot x)]\backslash x\Leftrightarrow \\
    g_1(x,y)=x\cdot \big(x\cdot g_1(x,yx\cdot x)\big).
      \end{gather*}
  \item $(Q,\cdot)$ is a {\Large 4}$_{(12\cdot 1)1=1\cdot (2\cdot 11)}^{3,1}$ loop if and only if $(xy\cdot x)x=x\cdot (y\cdot xx)$
\begin{gather*}
\Leftrightarrow (xy\cdot x)x\backslash x=x\cdot (y\cdot xx)\backslash x\Leftrightarrow x\backslash (x\backslash (x\backslash (x/y)))=g_1(x,y\cdot xx)\Leftrightarrow\\
g_1(x,y)=x\cdot \big(x\cdot g_1(x,y\cdot xx)\big).
\end{gather*}
  \item This is achieved by assuming the hypothesis that $(Q,\cdot)$ is a {\Large 3}$_{12\cdot 2=1\cdot 22}^{1,2}$ loop and using 2. and 3..
\end{enumerate}
\end{proof}

\begin{myth}\label{davidlem17}
Let $(Q,\cdot,\backslash,/)$ be a middle Bol loop and let $g_1:Q^2\rightarrow Q$ be defined as:
$g_1(x,y)=xy\backslash x~\textrm{or}~ g_1(x,y)=x\backslash (x/y)$.
\begin{enumerate}
   \item $(Q,\cdot)$ is a {\Large 5}$_{(12\cdot 1)1\cdot 1=1\cdot (21\cdot 1)1}^{4,1}$ loop if and only if $g_1(x,y)=x\big[x\cdot \big(x\cdot g_1(x,(yx\cdot x)x)\big)\big]$.
  \item $(Q,\cdot)$ is a {\Large 5}$_{(12\cdot 1)1\cdot 1=1\cdot (21\cdot 11)}^{4,1}$ loop if and only if $g_1(x,y)=x\big[x\cdot \big(x\cdot g_1(x,yx\cdot xx)\big)\big]$.
       \item $(Q,\cdot)$ is a {\Large 5}$_{(12\cdot 1)1\cdot 1=1\cdot (2\cdot 11)1}^{4,1}$ loop if and only if $g_1(x,y)=x\big[x\cdot \big(x\cdot g_1(x,(y\cdot xx)x)\big)\big]$.
           \item The following are equivalent:
           \begin{multicols}{2}
           \begin{enumerate}
             \item $(Q,\cdot)$ is a {\Large 5}$_{(12\cdot 1)1\cdot 1=1\cdot 2(11\cdot 1)}^{4,1}$ loop.
             \item $(Q,\cdot)$ is a {\Large 5}$_{(12\cdot 1)1\cdot 1=1\cdot 2(1\cdot 11)}^{4,1}$ loop.
             \item $(Q,\cdot)$ is a {\Large 5}$_{(12\cdot 1)1\cdot 1=1\cdot 2(111)}^{4,1}$ loop.
             \item $g_1(x,y)=x\big[x\cdot \big(x\cdot g_1(x,yx^3)\big)\big]$.
           \end{enumerate}
           \end{multicols}
  \item If $(Q,\cdot)$ is a {\Large 3}$_{12\cdot 2=1\cdot 22}^{1,2}$ loop, then the following are equivalent:
  \begin{multicols}{2}
  \begin{enumerate}
    \item $(Q,\cdot)$ is a {\Large 5}$_{(12\cdot 1)1\cdot 1=1\cdot (21\cdot 1)1}^{4,1}$ loop.
    \item $(Q,\cdot)$ is a {\Large 5}$_{(12\cdot 1)1\cdot 1=1\cdot (21\cdot 11)}^{4,1}$ loop.
    \item $(Q,\cdot)$ is a {\Large 5}$_{(12\cdot 1)1\cdot 1=1\cdot (2\cdot 11)1}^{4,1}$ loop .
  \end{enumerate}
  \end{multicols}
   \item If $(Q,\cdot)$ is a {\Large 3}$_{(1\cdot 22)2=1\cdot (22\cdot 2)}^{1,3}$ loop, then the following are equivalent:
  \begin{multicols}{2}
  \begin{enumerate}
    \item $(Q,\cdot)$ is a {\Large 5}$_{(12\cdot 1)1\cdot 1=1\cdot (2\cdot 11)1}^{4,1}$ loop.
    \item $(Q,\cdot)$ is a {\Large 5}$_{(12\cdot 1)1\cdot 1=1\cdot 2(11\cdot 1)}^{4,1}$ loop.
    \item $(Q,\cdot)$ is a {\Large 5}$_{(12\cdot 1)1\cdot 1=1\cdot 2(1\cdot 11)}^{4,1}$ loop .
  \end{enumerate}
  \end{multicols}
 \item If $(Q,\cdot)$ is a {\Large 3}$_{12\cdot22=1\cdot (2\cdot 22)}^{1,3}$ loop, the following are equivalent:
  \begin{multicols}{2}
  \begin{enumerate}
    \item $(Q,\cdot)$ is a {\Large 5}$_{(12\cdot 1)1\cdot 1=1\cdot (21\cdot 11)}^{4,1}$ loop.
    \item $(Q,\cdot)$ is a {\Large 5}$_{(12\cdot 1)1\cdot 1=1\cdot (2\cdot 11)1}^{4,1}$ loop.
    \item $(Q,\cdot)$ is a {\Large 5}$_{(12\cdot 1)1\cdot 1=1\cdot 2(11\cdot 1)}^{4,1}$ loop.
  \end{enumerate}
  \end{multicols}
   \item If $(Q,\cdot)$ is a {\Large 3}$_{12\cdot 2=1\cdot 22}^{1,2}$ loop, {\Large 3}$_{(1\cdot 22)2=1\cdot (22\cdot 2)}^{1,3}$ loop and {\Large 3}$_{12\cdot22=1\cdot (2\cdot 22)}^{1,3}$ loop, then the following are equivalent:
  \begin{multicols}{2}
  \begin{enumerate}
    \item $(Q,\cdot)$ is a {\Large 5}$_{(12\cdot 1)1\cdot 1=1\cdot (21\cdot 1)1}^{4,1}$ loop.
    \item $(Q,\cdot)$ is a {\Large 5}$_{(12\cdot 1)1\cdot 1=1\cdot (21\cdot 11)}^{4,1}$ loop.
    \item $(Q,\cdot)$ is a {\Large 5}$_{(12\cdot 1)1\cdot 1=1\cdot (2\cdot 11)1}^{4,1}$ loop.
    \item $(Q,\cdot)$ is a {\Large 5}$_{(12\cdot 1)1\cdot 1=1\cdot 2(11\cdot 1)}^{4,1}$ loop.
    \item $(Q,\cdot)$ is a {\Large 5}$_{(12\cdot 1)1\cdot 1=1\cdot 2(1\cdot 11)}^{4,1}$ loop.
    \item $(Q,\cdot)$ is a {\Large 5}$_{12\cdot 111=1\cdot 2(111)}^{4,1}$ loop.
  \end{enumerate}
  \end{multicols}
\item $(Q,\cdot)$ is a {\Large 6}$_{\big[(12\cdot 1)1\cdot 1\big]1=1\cdot \big[(21\cdot 1)1\big]1}^{5,1}$ loop if and only if

$g_1(x,y)=x\cdot x\big[x\cdot \big(x\cdot g_1(x,(yx\cdot x)x\cdot x)\big)\big]$.
\end{enumerate}
\end{myth}
\begin{proof}
This is similar to the proof of Theorem~\ref{davidlem16}.
\end{proof}

\begin{myth}\label{davidlem16.1}
Let $(Q,\cdot,\backslash,/)$ be a middle Bol loop and let $g_2:Q^2\rightarrow Q$ be defined as:
$g_2(x,y)=x\backslash yx~\textrm{or}~ g_2(x,y)=(y\backslash x)/x$.
\begin{enumerate}
  \item $(Q,\cdot)$ is a {\Large 3}$_{12\cdot 1=1\cdot 21}^{2,1}$ loop if and only if $g_2(x,y)=g_2(x,xy)\cdot x$.
  \item $(Q,\cdot)$ is a {\Large 4}$_{1\cdot (1\cdot 21)=(1\cdot 12)1}^{3,1}$ loop if and only if $g_2(x,y)=g_2(x,x\cdot xy)x\cdot x$.
  \item $(Q,\cdot)$ is a {\Large 4}$_{1\cdot (1\cdot 21)=(11\cdot 2)1}^{3,1}$ loop if and only if $g_2(x,y)=g_2(x,xx\cdot y)x\cdot x$.
  \item If $(Q,\cdot)$ is a {\Large 3}$_{11\cdot 2=1\cdot 12}^{2,1}$ loop, then the following are equivalent:
  \begin{multicols}{2}
  \begin{enumerate}
    \item $(Q,\cdot)$ is a {\Large 4}$_{(1\cdot 12)1=1\cdot (1\cdot 21)}^{3,1}$ loop.
    \item $(Q,\cdot)$ is a {\Large 4}$_{(11\cdot 2)1=1\cdot (1\cdot 21)}^{3,1}$ loop.
    \item $(Q,\cdot)$ is a {\Large 4}$_{11\cdot 21=(11\cdot 2)1}^{3,1}$ loop.
  \end{enumerate}
  \end{multicols}
\end{enumerate}
\end{myth}
\begin{proof}
This is similar to the proof of Theorem~\ref{davidlem16} with the aid of Equation~\eqref{davideq18af.1}.
\end{proof}

\begin{myth}\label{davidlem17.1}
Let $(Q,\cdot,\backslash,/)$ be a middle Bol loop and let $g_1:Q^2\rightarrow Q$ be defined as:
$g_1(x,y)=xy\backslash x~\textrm{or}~ g_1(x,y)=x\backslash (x/y)$.
\begin{enumerate}
   \item $(Q,\cdot)$ is a {\Large 5}$_{1\cdot 1(1\cdot 21)=1(1\cdot 12)\cdot 1}^{4,1}$ loop if and only if $g_2(x,y)=\big(g_2(x,x(x\cdot xy))x\cdot x\big)x$.
  \item $(Q,\cdot)$ is a {\Large 5}$_{1\cdot 1(1\cdot 21)=(11\cdot 12)1}^{4,1}$ loop if and only if $g_2(x,y)=\big(g_2(x,xx\cdot xy)x\cdot x\big)x$.
       \item $(Q,\cdot)$ is a {\Large 5}$_{1\cdot 1(1\cdot 21)=1(11\cdot 12)\cdot1}^{4,1}$ loop if and only if $g_2(x,y)=\big(g_2(x,x(xx\cdot y))x\cdot x\big)x$.
           \item The following are equivalent:
           \begin{multicols}{2}
           \begin{enumerate}
             \item $(Q,\cdot)$ is a {\Large 5}$_{1\cdot 1(1\cdot 21)=(1\cdot 11)2\cdot1}^{4,1}$ loop.
             \item $(Q,\cdot)$ is a {\Large 5}$_{1\cdot 1(1\cdot 21)=(11\cdot 1)2\cdot1}^{4,1}$ loop.
             \item $(Q,\cdot)$ is a {\Large 5}$_{1\cdot 1(1\cdot 21)=(111)2\cdot1}^{4,1}$ loop.
             \item $g_2(x,y)=\big(g_2(x,x^3y))x\cdot x\big)x$.
           \end{enumerate}
           \end{multicols}
  \item If $(Q,\cdot)$ is a {\Large 3}$_{1\cdot 12=11\cdot 2}^{2,1}$ loop, then the following are equivalent:
  \begin{multicols}{2}
  \begin{enumerate}
    \item $(Q,\cdot)$ is a {\Large 5}$_{1\cdot 1(1\cdot 21)=1(1\cdot 12)\cdot 1}$ loop.
    \item $(Q,\cdot)$ is a {\Large 5}$_{1\cdot 1(1\cdot 21)=(11\cdot 12)1}^{4,1}$ loop.
    \item $(Q,\cdot)$ is a {\Large 5}$_{1\cdot 1(1\cdot 21)=1(11\cdot 12)\cdot1}^{4,1}$ loop .
  \end{enumerate}
  \end{multicols}
   \item If $(Q,\cdot)$ is a {\Large 3}$_{1(11\cdot 2)2=(1\cdot 11)2}^{3,1}$ loop, then the following are equivalent:
  \begin{multicols}{2}
  \begin{enumerate}
    \item $(Q,\cdot)$ is a {\Large 5}$_{1\cdot 1(1\cdot 21)=1(11\cdot 2)\cdot1}^{4,1}$ loop.
    \item $(Q,\cdot)$ is a {\Large 5}$_{1\cdot 1(1\cdot 21)=(1\cdot 11)2\cdot1}^{4,1}$ loop.
    \item $(Q,\cdot)$ is a {\Large 5}$_{1\cdot 1(1\cdot 21)=(11\cdot 1)2\cdot1}^{4,1}$ loop .
  \end{enumerate}
  \end{multicols}
 \item If $(Q,\cdot)$ is a {\Large 3}$_{11\cdot12=(11\cdot 1)2}^{3,1}$ loop, the following are equivalent:
  \begin{multicols}{2}
  \begin{enumerate}
    \item $(Q,\cdot)$ is a {\Large 5}$_{1\cdot 1(1\cdot 21)=(11\cdot 12)1}^{4,1}$ loop.
    \item $(Q,\cdot)$ is a {\Large 5}$_{1\cdot 1(1\cdot 21)=1(11\cdot 2)\cdot1}^{4,1}$ loop.
    \item $(Q,\cdot)$ is a {\Large 5}$_{1\cdot 1(1\cdot 21)=(1\cdot 11)2\cdot1}^{4,1}$ loop.
  \end{enumerate}
  \end{multicols}
   \item If $(Q,\cdot)$ is a {\Large 3}$_{1\cdot 12=11\cdot 2}^{2,1}$ loop, {\Large 3}$_{1(11\cdot 2)2=(1\cdot 11)2}^{3,1}$ loop and {\Large 3}$_{11\cdot12=(11\cdot 1)2}^{3,1}$ loop, then the following are equivalent:
  \begin{multicols}{2}
  \begin{enumerate}
    \item $(Q,\cdot)$ is a {\Large 5}$_{1\cdot 1(1\cdot 21)=1(1\cdot 12)\cdot 1}^{4,1}$ loop.
    \item $(Q,\cdot)$ is a {\Large 5}$_{1\cdot 1(1\cdot 21)=(11\cdot 12)\cdot 1)}^{4,1}$ loop.
    \item $(Q,\cdot)$ is a {\Large 5}$_{1\cdot 1(1\cdot 21)=1(11\cdot 2)\cdot 1}^{4,1}$ loop.
    \item $(Q,\cdot)$ is a {\Large 5}$_{1\cdot 1(1\cdot 21)=(1\cdot 11)2\cdot 1)}^{4,1}$ loop.
    \item $(Q,\cdot)$ is a {\Large 5}$_{1\cdot 1(1\cdot 21)=(11\cdot 1)2\cdot 1}^{4,1}$ loop.
    \item $(Q,\cdot)$ is a {\Large 5}$_{111\cdot 21=(111)2\cdot 1}^{4,1}$ loop.
  \end{enumerate}
  \end{multicols}
\item $(Q,\cdot)$ is a {\Large 6}$_{1[1\cdot 1(1\cdot 21)]=1[1(1\cdot 12)]\cdot 1}^{5,1}$ loop if and only if

$g_2(x,y)=[[[g_2(x,(yx\cdot x(x(x\cdot xy))))x\cdot ]x]x]x$.
\end{enumerate}
\end{myth}
\begin{proof}
This is similar to the proof of Theorem~\ref{davidlem16}.
\end{proof}

\end{document}